\documentclass[leqno,12pt]{amsart}

\usepackage{amssymb}
\usepackage{amscd}  
\usepackage{mathrsfs}
\newtheorem{theorem}{Theorem}

\newtheorem{proposition}[theorem]{Proposition}
\newtheorem{lemma}[theorem]{Lemma}
\newtheorem{corollary}[theorem]{Corollary}

\theoremstyle{definition}

\newtheorem{remark}[theorem]{Remark}
\newtheorem{example}{Example}

\begin{document}

\title[On the canonical real structure on wonderful varieties]
{On the canonical real structure on wonderful varieties}
\author[D. Akhiezer and S. Cupit-Foutou ]{D. Akhiezer and S. Cupit-Foutou}
\thanks{Supported by SFB/TR 12,
{\it Symmetry and universality in mesoscopic systems}, of
the Deutsche Forschungsgemeinschaft}

\address{Dmitri Akhiezer, Institute for Information Transmission Problems, B.Karetny per. \ 19, 101447 Moscow, Russia}
\email{akhiezer@iitp.ru}

\address{St\'ephanie Cupit-Foutou, Ruhr-Universit\"at Bochum, NA 4/67, Bochum, 
Germany}
\email{stephanie.cupit@rub.de}

\begin{abstract}
We 
study equivariant real structures on spherical varieties.
We call such a structure canonical if it is equivariant
with respect to the involution defining the split real form 
of the acting reductive group $G$. 
We
prove the existence and uniqueness of a canonical structure
for homogeneous spherical varieties $G/H$ with $H$ self-normalizing
and for
their wonderful embeddings.
For a strict wonderful variety
we give an estimate of the number of real form orbits on the
set of real points. 
\end{abstract}
\maketitle

\tableofcontents
%------------------------------------------------------------------------

\section*{Introduction}

A real structure on a complex manifold $X$ is an anti-holomorphic involution $\mu : X \to X$.
The set of fixed points $X^\mu $ of $\mu $ is called the real part of $(X, \mu)$. 
If it is clear from the context which $\mu $ is considered then
the real part will be denoted by ${\mathbb R}X$. 
In our paper, we are interested in the algebraic case.
This means that $X$ is a complex algebraic variety, which we will assume non-singular though this
is not needed for the definition of a real structure. 
Also, $\mu $ is algebraic in the sense that for any function $f$ regular at $x \in X$
the function $\overline {f\circ \mu}$ is regular at $\mu (x)$.

It is not easy to classify all real structures on a given variety $X$.
Much work is done for compact toric varieties, where one has the notion of a toric real structure.  
Namely, if $X$ is a toric variety acted on by an algebraic torus $T$ then
a real structure $\mu : X \to X$ is said to be toric if
$\mu $ normalizes the $T$-action.
It is natural to classify toric real structures up to conjugation by toric
automorphisms, i.e.,
by automorphisms of $X$ normalizing the $T$-action. Again,
such a classification is not easy. 
For toric surfaces and threefolds it was obtained by C. Delaunay; see \cite{De}.

In the toric case, there is a notion of a canonical real structure.
This is a real structure which is usually defined as complex conjugation on the open $T$-orbit,
but we prefer a slightly different and more general definition. 
Let $\sigma : T \to T$  be the involutive anti-holomorphic automorphism of the real Lie group
$T$ which coincides with inversion on the maximal compact torus $T_c \subset T$. 
If $T \simeq ({\mathbb C}^*)^n$ then the real form defined by $\sigma $
is split, i.e., isomorphic to $({\mathbb R}^*)^n$. 
A canonical real structure on a toric variety $X$ is a real structure which satisfies
$$\mu (a\cdot x) = \sigma(a) \cdot \mu (x) \leqno{(*)}$$ 
for all $x \in X,\ a\in T$.
Of course, a canonical real structure is uniquely defined by the image of one point in the open orbit
and any two canonical real structures are related by 
$\mu ^\prime (x)=t\cdot \mu(x)$, where $t\in T$ and $\sigma (t)\cdot t = 1$.

Our goal is to generalize this notion to varieties acted on by reductive algebraic groups.
Let $G$ be a connected reductive algebraic group defined over ${\mathbb C}$.
We recall that an algebraic involution $\theta $ of $G$ is called a Weyl involution if $\theta (t) = t^{-1}$
for all $t$ in some algebraic torus $T \subset G$.  
Such an involution is known to be unique up to conjugation by an inner automorphism. 
By Cartan Fixed Point Theorem, one can always find
a maximal compact subgroup $K \subset G$, such that $\theta (K) = K$.
Then the corresponding Cartan involution $\tau $ commutes with $\theta$
and the product $\sigma = \tau \circ \theta = \theta \circ \tau $
is an involutive anti-holomorphic automorphism of $G$ defining the split real form.
Assume now that $G$ acts on an algebraic variety $X$.
Then a real structure $\mu : X \to X$ is called canonical if $\mu $
satisfies the above condition $(*)$ for all $x\in X,\ a \in G$. 
We remark that it suffices to check $(*)$ only for $a \in K$, in which case
one can replace $\sigma $ by $\theta$.

The most natural generalization of toric varieties to
the case of reductive algebraic groups is the notion
of spherical varieties, which we recall in Section~\ref{wonderful}.

Suppose $X$ is affine and non-singular.
For $X$ spherical a canonical real structure 
$\mu : X \to X$ always exists (\cite{A}, Theorem 1.2). 
However, in the non-spherical case such a 
structure may not exist even if $X$ 
is homogeneous (\cite{AP}, Proposition 6.3).

In this paper, we study the problem of existence
of a canonical real 
structure  
for all homogeneous spherical varieties, affine or not affine.
We also
consider the similar question for some complete spherical varieties, 
namely the so-called wonderful varieties. 
The definition of wonderful and strict wonderful varieties is recalled in 
Section~\ref{wonderful}.

We start with a finiteness theorem for real structures on wonderful varieties 
(Theorem~\ref{fin}).
Then we prove some topological properties of a canonical real
structure on a wonderful variety provided such a structure exists
(Theorem~\ref{non-emptyness and connectedness}).
After that we show that a canonical
real structure exists and is uniquely defined for 
homogeneous spherical varieties $G/H$ with
$H$ self-normalizing (Theorem~\ref{canonicalstructure}) and for their wonderful
completions (Theorem~\ref{canonicalwonderful}). 
As an application we show that for a spherical subgroup $H \subset G$,
whose normalizer is self-normalizing, there is always an anti-holomorphic
involution $\sigma : G \to G$, defining the split real form
and such that $\sigma (H) = H$ (Theorem~\ref{application}).
Finally, we give an estimate of the total number of real form
orbits in ${\mathbb R}
X$ for the canonical real structure on a strict wonderful variety $X$
(Theorem~\ref{estimationrealorbits}).
    
\section{Wonderful varieties}\label{wonderful}

Recall that $G$ is a connected reductive algebraic group over ${\mathbb C}$. A normal algebraic $G$-variety $X$ is called spherical if $X$ contains an open orbit of a Borel subgroup $B \subset G$.
We denote the open orbits of $B$ and $G$ on $X$ by $X_B^\circ $ and $X_G^\circ $ respectively. 

The following definition is due to D. Luna (\cite{Lu96}).
An algebraic $G$-variety $X$ is called \textsl{wonderful} if
\smallbreak\noindent
{\rm(i)}\enspace
$X$ is complete and smooth;
\smallbreak\noindent
{\rm(ii)}\enspace
$X$ admits an 
open $G$-orbit whose complement 
consists of a finite union of smooth prime divisors $X_1\ldots, X_r$ with
normal crossings;
\smallbreak
\noindent
{\rm(iii)}\enspace
the $G$-orbit closures of $X$ 
are given by the partial intersections of the $X_i$'s.

Remark that a wonderful variety $X$ has a unique closed $G$-orbit.
The latter is the full intersection of the boundary divisors $X_i$ of $X$.

D. Luna proved that wonderful $G$-varieties are spherical.
The connected center of $G$ acts trivially on a wonderful variety, so if
$G$ acts effectively then $G$ is semisimple. If a spherical homogeneous space
$G/H$ admits an equivariant wonderful embedding then such an
embedding is unique up to a $G$-isomorphism; see ~\cite{Lu96}
and references therein.

By a theorem of F. Knop, 
a wondferful equivariant embedding of
$G/H$ always exists if the spherical subgroup $H$ is self-normalizing in $G$;
see \cite{K}.

\begin{proposition}\label{norm} 
Let $X$ be a wonderful variety and let $X_G^\circ = G/H$.
Then $H$ has finite index in its normalizer.
\end{proposition}
\begin{proof}
See Section 4.4 in~\cite{Br97}.
\end{proof}

A wonderful variety is called \textsl{strict} if each of its points
has a self-normalizing stabilizer.
The class of strict
wonderful varieties includes flag varieties and De\,Concini-Procesi compactifications (\cite{DP83}). 
Strict wonderful varieties are classified in~\cite{BCF}. 

For any variety $X$, let ${\rm Aut}(X)$ denote the automorphism group of $X$.
We will need the following proposition describing the identity
component ${\rm Aut}_0(X)$ for a wonderful $G$-variety $X$.

\begin{proposition}[\cite{Br2}, Theorem\,2.4.2] \label{ss} 
If $X$ is wonderful under $G$ then ${\rm Aut}_0(X)$
is semisimple and $X$ is wonderful under the action of ${\rm Aut}_0(X)$.
\end{proposition}
 
In addition, we have the following proposition, for which we could not find a reference.

\begin{proposition}
Let $X$ be a wonderful variety. 
Then ${\rm Aut}_0(X)$ has finite index in ${\rm Aut}(X)$.
\end{proposition}

\begin{proof}
Write $X_G^\circ = G/H$,
where $H$ is the stabilizer of a point $x_0 \in X_G^\circ $. Let $N$
be the normalizer of $H$ in $G$. By Proposition~\ref{norm} the orbit
$N\cdot x_0$ is finite.
For any $\alpha \in {\rm Aut}(X)$ and $g\in {\rm Aut}_0(X)$
put 
$$\iota_\alpha (g) = \alpha \cdot g \cdot \alpha ^{-1}.$$
Let $L$ denote the group of all automorphisms of the group ${\rm Aut}_0(X)$.
Then we have the homomorphism 
$$\varphi: {\rm Aut}(X) \to L,\ \ \varphi (\alpha ) = \iota _\alpha,$$
whose image contains the group of inner automorphisms of
${\rm Aut}_0(X)$. Since the latter group is semisimple by Proposition~\ref{ss},
${\rm Im}(\varphi)$ has finitely many connected components.  
We now prove that ${\rm Ker}(\varphi )$ is finite. It then
follows that ${\rm Aut}(X)$ has finitely many connected components.

If $\alpha \in {\rm Ker}(\varphi)$ then $\alpha $ commutes with
all automorphisms from $G$. Thus $\alpha (gx_0) = g\alpha (x_0)$ 
for all $g\in G$. Since $X$ has only one open $G$-orbit, we
have $\alpha (X_G^\circ) = X_G^\circ $ and,
in particular, $\alpha (x_0) = ax_0$ for some $a\in G$.
Now take $g\in H$. Then
$$ax_0 = \alpha (x_0) = \alpha (gx_0) = g\alpha (x_0) = gax_0,$$
hence $aga^{-1} \in H$ and $a \in N$. Since the $N$-orbit of $x_0$
is finite, there are only finitely many possibilities for $ax_0$.
But, for $\alpha (x_0)$ fixed, $\alpha $
is uniquely determined on the open $G$-orbit and thus everywhere on $X$.
\end{proof}

\section{Finiteness theorem}

The group ${\rm Aut}(X)$ acts on the set of real structures on $X$ by
$$
\mu \mapsto \alpha \cdot \mu \cdot \alpha ^{-1}. 
$$
For $X$ wonderful,
we prove that this action has only finitely many orbits.

\begin{theorem}\label{fin}
Let $X$ be a wonderful variety. Then, up to an automorphism of $X$,
there are only finitely many real structures on $X$.
\end{theorem}

\begin{proof}

Assume that $X$ has at least one real structure $\mu _0 $.
Then 
${\rm Aut}(X)$-orbits on the set of real structures on $X$ are in one-to-one
correspondence with the cohomology classes from $H^1({\mathbb Z}_2, {\rm Aut}(X))$,
where the generator $\gamma \in {\mathbb Z}_2$ acts on ${\rm Aut}(X)$ by sending
$\alpha$ to  $\mu _0 \alpha \mu _0$.
We now use the exact cohomology sequence,
associated with the normal subgroup ${\rm Aut}_0(X) \triangleleft
{\rm Aut}(X)$. From Corollary 3 in I.5.5 of \cite{S}
it follows that $H^1({\mathbb Z}_2, {\rm Aut}(X))$
is finite if the following two conditions are fulfilled:

(1) ${\rm Aut}(X)/{\rm Aut}_0(X)$ is finite;

(2) for the ${\mathbb Z}_2$-action 
on ${\rm Aut}_0(X)$ obtained by twisting the given action by
an arbitrary cocycle $a \in Z^1({\mathbb Z}_2, {\rm Aut}(X))$ 
the corresponding cohomology 
set
$H^1({\mathbb Z}_2, \,_a{\rm Aut}_0(X))$ is finite.

\noindent
We have just proved (1). Since ${\rm Aut}_0(X)$ is linear algebraic,
(2) follows from Borel-Serre's Theorem (see \cite{BS}).
\end{proof}

\section{General properties of equivariant real structures}
                      
Let $X$ be a non-singular complex algebraic variety with a
real structure $\mu : X \to X$. Suppose $G$ is a connected algebraic
group acting on $X$ and let $\sigma : G \to G$ be an involutive
anti-holomorphic
automorphism of $G$ as a real algebraic group.
Then the fixed point subgroup
$$
G^\sigma  = \{g \in G \,\vert \, \sigma(g) = g\}
$$ 
is real algebraic and its identity component $G^\sigma_0$
is a closed real Lie subgroup in $G$.

We call $\mu $ a $\sigma$-equivariant real structure if
$$
\mu(g\cdot x) = \sigma(g)\cdot \mu (x)\quad \mbox{ for all $g\in G,\ x\in X$.}
$$
Later on, we will be interested in the case when
$G^\sigma $ is a split real form of a reductive
group $G$;
see Introduction. However, in the following elementary lemma
$G$ and $\sigma $ are arbitrary.

\begin{lemma}\label{dimension}
Let $H \subset G$ be an algebraic subgroup and let $X = G/H$. Suppose $x_0
\in {\mathbb R}X$. Then the connected component of ${\mathbb R}X$
through $x_0$ coincides with $G^\sigma_0 \cdot x_0$. The orbit
$G^\sigma _0 \cdot x_0 $ is Zariski dense in $X$.
\end{lemma}

\begin{proof}
Let $n$ be the complex dimension of $X$. Then the real dimension of
${\mathbb R}X$ is also $n$. Since $\mu $ is $\sigma $-equivariant,
we have $G^\sigma (x_0) \subset {\mathbb R}X$. Thus it suffices to show that
${\rm dim}\,G^\sigma _0(x_0) \ge n$. Let $G_{x_0}$ be the
stabilizer of $x_0$ in $G$. Then
$G^\sigma _0\cap G_{x_0}$ is a totally real submanifold in $G_{x_0}$,
hence
$${\rm dim}_{\mathbb R} \, G^\sigma _0 \cap G_{x_0} \le {\rm dim}_{\mathbb C}
\, G_{x_0},$$
and so we obtain
$${\rm dim} \, G^\sigma _0 \cdot x_0 = {\rm dim}\, G^\sigma _0
- {\rm dim}\, G^\sigma _0 \cap G_{x_0}
\ge {\rm dim}_{\mathbb C}\, G - {\rm dim}_{\mathbb C}\, G_{x_0} = n.$$
Finally, 
since 
$G^\sigma _0 \cdot x_0 \subset X$ is a totally real submanifold
of maximal possible dimension, $G^\sigma _0 \cdot x_0 $
is not contained in an algebraic
subvariety of dimension smaller than $n$.

\end{proof}

From now on $G$ is reductive.
We need some preparatory lemmas on the involution
$\sigma: G \to G$ defining the split real form of $G$.
We also fix some notation, which will be used all the time
in the sequel. So let $T \subset G$ be a torus, on which 
$\sigma $ acts as the involutive anti-holomorphic
automorphism with fixed point subgroup being the non-compact
real part of $T$. In coordinates, if $T \simeq ({\mathbb C}^*)^r$
then $$\sigma (z_1,\ldots,z_r) = (\bar z_1, \ldots, \bar z_r),\ \ 
z=(z_1,\ldots,z_r) \in ({\mathbb C}^*)^r.$$

\begin{lemma}\label{transformedweight}
Let $\chi$ be a character of $T$. Then $\overline{\chi\circ\sigma}=\chi$.
\end{lemma}

\begin{proof}
Take $t$ in the non-compact real part of $T$.
Then $\sigma (t) = t$ and the value of $\chi$ is real.
This shows that the weights $\overline{\chi\circ\sigma}$ and $\chi$ coincide
on real points, hence also everywhere by analytic extension.
\end{proof}

\begin{lemma}\label{rootspaces} Let $\mathfrak g$ be the Lie algebra
of $G$.
Denote the associated involution of $\mathfrak g$ again by $\sigma$.
Then all root spaces  
in $\mathfrak g$ are $\sigma$-stable.
\end{lemma}

\begin{proof}
Let $\alpha : T \to {\mathbb C}^*$
be a root, 
$\mathfrak g_\alpha $ the corresponding root space, 
$X_\alpha \in {\mathfrak g}_\alpha$, and $t\in T$. Then
$${\rm Ad}(t)\cdot X_{\alpha}
 = \alpha (t) X_{\alpha}$$ implies 
$$
{\rm Ad}(\sigma (t))\cdot \sigma(X_\alpha) = \overline{\alpha (t)} \, 
\sigma (X_\alpha)$$
or, equivalently,
$${\rm Ad}(t) \cdot \sigma(X_\alpha ) = \overline{
\alpha \circ \sigma (t)} \, \sigma (X_\alpha ) = \alpha (t)
\, \sigma (X_\alpha ),$$
where the last equality follows from
Lemma~\ref{transformedweight}. Therefore
$\sigma (X_\alpha ) \in {\mathfrak g}_\alpha$,
showing that the root spaces are $\sigma$-stable. 
\end{proof}

\begin{corollary}\label{borel}
With the above choice of $T$ we have $\sigma (B) = B$ and $\sigma (P) = P$
for any Borel subgroup $B\subset G$ containing $T$
and any parabolic subgroup $P\subset G$ containing $B$.
\end{corollary}

\begin{proof} The Lie algebras of $\mathfrak p$ and $\mathfrak b$
are spanned by root spaces and the Lie algebra $\mathfrak t$ of $T$,
so Lemma~\ref{rootspaces} applies.
\end{proof}

We will assume throughout the paper that $T, B$ and $P$
are chosen as in Corollary~\ref{borel}.

\begin{proposition}\label{flags}
With the above choice of $P$, 
define a self-map of the flag variety $X = G/P$ by
 $\mu (g\cdot P) = \sigma(g)\cdot P$.
Then $\mu $ is a $\sigma $-equivariant real structure on $X$.
Such a structure is uniquely defined.
The set ${\mathbb R}X$ is the unique closed $G^\sigma _0$-orbit on $X$.
In particular, the (possibly disconnected) real form
$G^\sigma $ is transitive on ${\mathbb R}X$.

\end{proposition}

\begin{proof} Clearly, the map $\mu $ is correctly defined,
anti-holomorphic,
$\sigma $-equivariant, and involutive. 
If there is another $\sigma $-equivariant real structure $\mu ^\prime$
on $X$, then the product $\mu ^\prime \cdot \mu$ is an automorphism
of $X$ commuting with the $G$-action. Since $P$ is self-normalizing,
such an automorphism is the identity map, hence $\mu ^\prime = \mu$.
By construction,
the base point $e\cdot P$ is contained in ${\mathbb R}X$.
According to Lemma~\ref{dimension},
each connected component of ${\mathbb R}X$
is a closed $G^\sigma_0$-orbit. By
\cite{W} such an orbit is unique,
so ${\mathbb R}X$ is connected and coincides with that orbit.
The last assertion is now obvious.
\end{proof}

For a wonderful variety $X$, the existence of a
$\sigma$-equivariant real structure requires some work involving Luna-Vust
invariants of spherical homogeneous spaces.
We postpone this until the next section.
Here, assuming that such a structure $\mu $ exists, we study
geometric properties of ${\mathbb R}X$.
The notation is as in Section 1. In particular,
$$Y = X_1 \cap \ldots \cap X_r$$
is the unique closed $G$-orbit in $X$. Note that $\mu (Y) = Y$
and ${\mathbb R} X \cap Y$ is the unique
closed $G_0^\sigma $-orbit in $Y$ by Proposition~\ref{flags}.

\begin{theorem}\label{non-emptyness and connectedness}
Let $X$ be any wonderful $G$-variety equipped with 
a $\sigma$-equivariant real structure $\mu $. 
Then:
\smallbreak
{\rm (i)}\enspace $G^\sigma_0$ has finitely many orbits on ${\mathbb R}X$;
\smallbreak
{\rm (ii)}\enspace ${\mathbb R}X \cap Gx \ne \emptyset $ for any $x \in X$;
\smallbreak
{\rm (iii)}\enspace there is exactly one closed $G^\sigma _0$-orbit in ${\mathbb R}X$;
this orbit is contained in the closed $G$-orbit and is $G^\sigma $-homogeneous; 
\smallbreak
{\rm (iv)}\enspace ${\mathbb R}X$ is connected.
\end{theorem}
                                                                
\begin{proof}
\smallskip                                              
{\rm (i)} ${\mathbb R}X$ is a non-empty
real algebraic set. In particular, ${\mathbb R}X$
has finitely many connected components. By Lemma~\ref{dimension},
each of them is one $G^\sigma _0$-orbit.

\smallskip
{\rm (ii)}  
We choose $B$ as in Corollary~\ref{borel}.
We prove first that $\mu $ preserves $G$-orbits. This is clear
for the open orbit, 
because its $\mu $-image is also an open orbit which is unique. 
Since the orbit structure is well understood
(see Section 1),
it is enough to prove that $\mu (X_i) = X_i$,
where $X_i$ are the boundary divisors.
Equivalently, it suffices to prove
that the $G$-invariant valuation $v_i$ centered on $X_i$ is
$\mu$-invariant in the sense that
$$v_i(\overline {f\circ \mu}) = v_i(f)$$
for any $f \in {\mathbb C}(X)\setminus \{0\}$.
It is enough to check this on $B$-eigenfunctions 
(see Appendix ~\ref{LV-invariants}),
but then Lemma~\ref{transformedweight} yields the
required equality.

Now, let $G\cdot x$ be any $G$-orbit on $X$ and let ${\rm cl}(G\cdot x)$
be its Zariski closure within $X$. 
Since $G\cdot x$ is $\mu$-stable, so is ${\rm cl}(G\cdot x)$.
Note that $Y \subset {\rm cl}(G\cdot x)$,
therefore $Z:= {\mathbb R}X \cap {\rm cl}(G\cdot x) \ne \emptyset$.
Furthermore, ${\rm cl }(G\cdot x)$ is a non-singular variety
and $Z \subset {\rm cl}(G\cdot x)$ is a totally real
submanifold of maximal possible dimension. Therefore $Z$
is not contained in the boundary ${\rm cl }(G\cdot x) \setminus G\cdot x$.

\smallskip
{\rm (iii)} Given a closed orbit $G^\sigma_0 \cdot y \subset {\mathbb R}X$,
consider the orbit $G^\sigma \cdot y$, which is also closed, and take
a fixed point of the real form $B^\sigma \subset B$ thereon. 
The existence of such a point                                
follows from Borel's theorem for connected split solvable groups.
Assuming $B^\sigma \cdot y = y$, we also
have $B\cdot y = y$.
But then $G\cdot y$ is projective, i.e., $G\cdot y = Y$. Thus our statement
is reduced to the case of flag varieties, and we can
apply Proposition ~\ref{flags}.

\smallskip
{\rm (iv)} Assume ${\mathbb R}X$ is disconnected. Since ${\mathbb R}X \cap Y$
is connected, we can find a connected component $W$ of 
${\mathbb R}X$, such that $W \cap Y = \emptyset $. Then 
$W$ is a closed $G^\sigma_0$-orbit and $G^\sigma \cdot W$
is a closed $G^\sigma $-orbit, which
also has empty intersection with $Y$. On the other hand,
by the above argument, $B^\sigma $
has a fixed point on $G^\sigma \cdot W$. Since that
fixed point is also fixed by $B$, it belongs to the
closed $G$-orbit $Y$. We get a contradiction showing that
${\mathbb R}X$ is in fact connected.   
 
\end{proof}

\section{The canonical real structure}

Recall that $T$ and $B$ are chosen as in Corollary~\ref{borel}.

\begin{proposition}\label{stableinvariants}
Any spherical subgroup $H \subset G$ is conjugate to $\sigma(H)$ by an
inner automorphism of $G$.
\end{proposition}

\begin{proof}
We note first that $\sigma(H)$ is a spherical subgroup of $G$.
We shall prove 
that the Luna-Vust invariants attached to $X_1 = G/H$ 
and $X_2 = G/\sigma(H)$ are the same
and then use the theorem in
Appendix~\ref{LV-invariants}, from where we also take the notations.

Consider the map $\mu: X_1 \to X_2$ defined by
$$X_1 \ni g\cdot H \ {\buildrel \mu \over\mapsto}\ 
 \sigma (g)\cdot \sigma(H) \in X_2.$$
We show that $\mu $ defines a bijection between the sets
of $B$-eigenfunctions on $X_2$ and $X_1$. Moreover,
the associated
map ${\Lambda}^+(X_2) \longrightarrow {\Lambda}^+(X_1)$
is the identity map.
Namely, let $f$ be a $B$-eigenfunction in $\mathbb C(X_2)$ 
and let $\lambda$ be its $B$-weight.
Then the complex conjugate of $f\circ \mu$ is a
$B$-eigenfunction in $\mathbb C(X_1)$ with
weight $\overline{\lambda\circ\sigma}$.
The latter is equal 
to $\lambda$ by Lemma~\ref{transformedweight}.
Since we can apply the same argument to the map $\mu ^{-1}$, 
it follows that the weight lattices
of $X_1$ and $X_2$ coincide and $\mu $ induces the identity
map on ${\Lambda}^+(X_2) = {\Lambda}^+(X_1)$. 

Further, consider the map 
$\mathcal V(X_1)\rightarrow\mathcal V(X_2)$, defined by 
$$v\mapsto (f\mapsto v({\overline {f\circ\mu}})).$$ 
This map is obviously bijective. Namely, 
its inverse is defined analogously by means of the mapping 
$\mu^{-1}: X_2 \rightarrow X_1$.

Finally, there is a natural bijection
$\iota : {\mathcal D}_{G,X_1} \to {\mathcal D}_{G,X_2}$ sending $D$
to $\pi_2[\sigma (\pi _1^{-1}(D)]$ where $\pi_1$ and $\pi _2$ are
the projections from $G$ to $X_1 $ and $X_2$ respectively. For this mapping, 
$\varphi _{\iota (D)}$ evaluated on $\overline {\lambda\circ\sigma }$
gives the same result as $\varphi _D$ evaluated on $\lambda$. By  Lemma~\ref{transformedweight},
$\overline {\lambda\circ\sigma}$ coincides with $\lambda$, and so we have
$\varphi _{\iota (D)} = \varphi _D$.
Similarly, $G_{\iota (D)} = \sigma (G_D) = G_D$
because $G_D$ is a parabolic subgroup containing $B$.
\end{proof}

\begin{theorem}\label{canonicalstructure}
Let $H$ be a spherical subgroup of 
$G$ and $a\in G$ such that $\sigma(H)=aHa^{-1}$.
The assignment
$$
\mu_0:gH\mapsto \sigma(g)aH
$$
defines an anti-holomorphic $\sigma $-equivariant
diffeomorphism of $G/H$.
If $H$ is self-normalizing then this map 
is involutive, hence a $\sigma $-equivariant real structure on $G/H$.
Furthermore, for $H$ self-normalizing a $\sigma $-equivariant 
real structure on $G/H$ is uniquely defined.
\end{theorem}

\begin{proof}
The first assertion follows from Proposition~\ref{stableinvariants}.
Further, since $\sigma$ is an involution of $G$, 
$\sigma(a)a$ belongs to the normalizer of $H$ in $G$.
The latter coincides with $H$. This proves the second assertion.
The product of two $\sigma $-equivariant
real structures on $G/H$ is
an automorphism of $G/H$ commuting with the $G$-action.
For $H$ self-normalizing in $G$ such an
automorphism is the identity map, and the last assertion follows.
\end{proof}

\begin{theorem}\label{canonicalwonderful}
Let $H$ be a self-normalizing 
spherical subgroup of $G$ and let $X$ be the wonderful completion of $G/H$.
Then there exists one and only one $\sigma$-equivariant real structure of $X$.
\end{theorem}

\begin{proof}
Let $\iota:G/H \rightarrow X$ be the given 
wonderful completion and  
let $\bar\iota: G/H \rightarrow \bar X$ 
be the corresponding anti-holomorphic map 
with $\bar X$ being the complex conjugate of $X$.
Recall that 
$\bar X=X$ as sets and 
that the sheafs of regular functions of $\bar X$ and $X$ are complex conjugate.

We endow $\bar X$ with the $G$-action 
$(g,x)\mapsto \sigma(g)\cdot x$, 
where $(g,x)\mapsto g\cdot x$ is the given action of $G$ on $X$. 
Note that this new action is regular on $\bar X$.

Consider the real structure $\mu_0$ introduced in 
Theorem ~\ref{canonicalstructure}.
Then $\bar\iota\circ\mu_0$ is again a wonderful completion of $G/H$. 
Since two wonderful completions of $G/H$ are $G$-isomorphic, 
there exists a $G$-isomorphism $\mu:X\rightarrow \bar X$ such that
$\mu\circ\iota=\bar\iota\circ\mu_0$. 
The map $\mu$ defines a $\sigma$-equivariant real structure on $X$.

Finally, a $\sigma $-equivariant real structure on $X$ is defined
by its restriction to the open $G$-orbit in $X$. The restriction
is unique by Theorem~\ref{canonicalstructure}. 
\end{proof}

In the remainder, the real structure defined in 
Theorem~\ref{canonicalwonderful} is called 
\textsl{ the canonical real structure of} $X$.
We want to give here a group-theoretical application
of Theorem~\ref{canonicalwonderful}.

\begin{theorem}\label{application}
If $H \subset G$ is a spherical subgroup with self-normalizing 
normalizer then there exists
an anti-holomorphic involution $\sigma: G \to G$,
defining the split real form and such that $\sigma (H) = H$.
Moreover, one can find a Borel subgroup
$B \subset G$, such that $B\cdot H$ is open in $G$
and $\sigma (B) = B$.
\end{theorem}

\begin{proof}
Let $N$ be the normalizer of $H$ in $G$. We start with some $\sigma $ and
take $a\in G$ as in Theorem~\ref{canonicalstructure}, i.e.,
$\sigma (H) = aHa^{-1}$. Then, of course,
$\sigma (N) = aNa^{-1}$.
Let $X$ be a wonderful equivariant completion of $G/N$
and let $\mu $ be the canonical $\sigma $-equivariant real structure on $X$.
By Theorem ~\ref{non-emptyness and connectedness}
we can find a $\mu$-fixed point in the open orbit.
Let $\mu(g_0\cdot N) = g_0\cdot N$.
Replace $\sigma $ by $\sigma _1 = i^{-1}_{g_0} \sigma i_{g_0}$,
where $i_{g_0}$ is the inner automorphism of $G$ 
given by $x \mapsto g_0 x g_0^{-1}$.
Also, replace $\mu $ by $\mu_1 =g_0^{-1}\mu g_0$.
A straightforward calculation shows that
$\mu_1 (gx)=
\sigma_1 (g)\mu_1 (x)$ for all $g\in G, \, x\in X$,
i.e., $\mu_1$ is a $\sigma _1$-equivariant real structure on $X$.
Moreover, for the new pair $(\mu_1, \sigma _1)$ we have
$$\mu _1(e\cdot N) = (g_0^{-1}\mu g_0)(e\cdot N) = g_0^{-1}\mu(g_0\cdot N)
=e\cdot N.$$   
Comparing the stabilizers at $e\cdot N$ and $\mu _1(e\cdot N)$,
we get
$$\sigma _1(N) = N.$$
It follows that
$$N = \sigma _1(N) = i^{-1}_{g_0} \, \sigma \, i_{g_0}\,(N) 
= g^{-1}_0\sigma(g_0)\,\sigma(N)\,\sigma(g_0)^{-1}g_0.$$
As we have seen, $\sigma (N) = aNa^{-1}$. Substituting
this in the previous equality, we get
$g_0^{-1}\sigma (g_0)a \in N$,
and it follows that 
$\sigma _1(H) = H$.

Now, assuming $\sigma (H) = H$ consider the subset $\Omega
\subset G/B$ 
whose points correspond to the Borel subgroups $B_* \subset G$ 
with $B_*\cdot H$ open in $G$.
Then $\Omega $ is Zariski open and
$\sigma $-stable. The subset of $\sigma $-fixed Borel subgroups
is a totally real submanifold in $G/B$, having maximal possible dimension.
Thus its intersection with $\Omega $ is non-empty. 
\end{proof}

\begin{remark}
The normalizer of a spherical subgroup is in general
not self-normalizing, see Example 4 in~\cite{Av}.
In Theorem~\ref{application}, we do not know if the condition of $N$ being
self-normalizing is essential.
\end{remark}

\section{Real part: local structure and $G^\sigma _0$-orbits}
Let 
$X$ be a strict wonderful 
$G$-variety of rank $r$ equipped with the canonical real structure $\mu$.
For a complex vector space $V$
and an anti-linear map $\nu : V \to V$ we denote by the same letter $\nu $
the induced anti-holomorphic map of ${\mathbb P}(V)$.  

\begin{proposition}\label{realembedding} There exist a simple $G$-module $V$
with the associated representation $\rho: G \to {\rm GL}(V)$,
an anti-linear involutive map $\nu : V \to V$,
and an embedding $\varphi : X \to {\mathbb P}(V)$, such that
\smallbreak \noindent
$({\rm i})\ \ \nu (\rho (g)\cdot v) = \rho(\sigma (g))\cdot \nu (v)\ \ \ 
(v\in V), $
\smallbreak \noindent
$({\rm ii})\ \ \varphi (gx) = \rho (g)\cdot \varphi (x)\ \ \ (x\in X) $ and
\smallbreak
\noindent 
$({\rm iii})\ \ \varphi (\mu x) = \nu \varphi (x)\ \ \ (x\in X).$ 
\smallbreak \noindent
In particular, $\mathbb R X$ is $G^\sigma$-equivariantly embedded into 
the real projective space ${\mathbb R}{\mathbb P}(V) \subset {\mathbb P}(V)$,
defined by $\nu $.
\end{proposition}

\begin{proof}
Since $X$ is a non-singular projective $G$-variety, 
$X$ can be $G$-equivariantly embedded into the projectivization of a $G$-module.
Let $\varphi: X\rightarrow\mathbb P(V)$ be such an embedding and 
let $\rho:G\rightarrow \mathrm{GL}(V)$ 
denote the representation associated to $V$.
Since $X$ is a strict wonderful variety, we may choose $V$ to be simple;
see ~\cite{P}.

Now, equip the complex conjugate vector space $\bar V$ 
with the $G$-module structure given by 
$g\mapsto \overline{\rho(\sigma(g))}$.
By Lemma~\ref{transformedweight}, it follows that the $G$-modules $V$ and $\bar V$ are isomorphic.
In other words, we have an anti-linear map $\nu : V \to V$ satisfying 
${(\rm i)}$.
Though $\nu $ is not necessarily involutive, we can modify $\nu $
to get this property. As in Appendix~\ref{localstructure},
let $v^-$ be a lowest weight vector
of $V$.
Then $\nu (v^-)$ is also a lowest weight vector, hence $\nu (v^-) = cv^-$
for some $c \in {\mathbb C}^*$. This implies $\nu ^2(v^-) = \nu (cv^-)=
\bar c \cdot \nu(v^-) = \vert c \vert ^2 v^-$. Replacing $\nu $ by
$\nu / {\vert c \vert}$, we get an involutive anti-linear map
satisfying $({\rm i})$.

Since $({\rm ii})$ is clear from the construction,
it remains to show $({\rm iii})$.
Note that
$\nu \circ \varphi \circ \mu$ is another $G$-equivariant embedding of $X$ into 
${\mathbb P}(V)$.
Thus $({\rm iii})$ follows from
the uniqueness of such an embedding; see ~\cite{P}.
\end{proof}

Let $Z$ be the slice defined in Appendix~\ref{localstructure}.
We show that $Z$
can be chosen to be
$\mu $-stable. As we have seen in Proposition~\ref{realembedding},
the line ${\mathbb C}\cdot v^-$ is $\nu $-stable.
So we may assume that $\nu(v^-) = v^-$. Then the tangent space 
$W:= T_{v^-}G\cdot v^-$ is also $\nu $-stable. Consider the real vector space
${\mathbb R}V = \{v\in V \ \vert \ \nu (v) = v\}$ and let
${\mathbb R}W = W\cap {\mathbb R}V$. Then ${\mathbb R}W$ is stable
under $L^\sigma $ and, also, under the Lie algebra $\mathfrak l ^\sigma$ 
of $L^\sigma $.
Now, the center of $\mathfrak l ^\sigma$ is 
contained in the center of the complexified algebra 
$\mathfrak l = \mathfrak l ^\sigma
\otimes {\mathbb C}$ and is therefore
represented by 
semisimple endomorphisms of ${\mathbb R}V$. The 
complete reducibility theorem for reductive Lie algebras over ${\mathbb R}$
implies that ${\mathbb R}W$ has a $\mathfrak l ^\sigma$-stable complement
in ${\mathbb R}V$; see~\cite{C}, Ch.IV, $\S$ 4. Call this
complement $E_{\mathbb R}$. The complexification
$E_{\mathbb R} \otimes {\mathbb C} \subset V$ is $\mathfrak l$-stable
and therefore $L$-stable. So we can take 
$E = E_{\mathbb R}\otimes {\mathbb C}$. 
Note that $E_{\mathbb R} = E \cap {\mathbb R}V$ is 
not just $\mathfrak l ^\sigma $-stable, but also $L^\sigma $-stable
even if $L^\sigma $ is disconnected.

Obviously, $\nu (E) = E$.
Furthermore, 
the linear form $\eta $ in Appendix~\ref{localstructure} can be chosen
real. Therefore, using
$({\rm iii})$ of Proposition~\ref{realembedding}, we see that $\mu (Z) = Z$.
Note 
that $P^u\cdot Z$ is $\mu$-stable and
${\mathbb R}(P^u \cdot Z)= (P^u)^\sigma \cdot {\mathbb R}Z$. 	

The first assertion of the following proposition is a real analogue
of Local Structure Theorem in~\cite{BLV}; 
see also Appendix~\ref{localstructure}.

\begin{proposition}\label{reallocalstructure}
\smallbreak\noindent
{\rm (i)}\enspace
 The natural mapping 
$$(P^u)^\sigma\times (\mathbb R Z)\rightarrow (P^u)^\sigma
\cdot {\mathbb R}Z = \mathbb R ( P^u \cdot  Z )
$$
is a $(P^u)^\sigma$-equivariant isomorphism.
\smallbreak\noindent
{\rm (ii)}\enspace
Each $G^\sigma_0$-orbit in $\mathbb R X$ contains points of the slice $Z$.
\end{proposition}

\begin{proof}
The first assertion follows readily from Local Structure Theorem 
and the above construction of $Z$.

To prove ${\rm (ii)}$, take a point $x \in {\mathbb R}X$.
Since $X$ is wonderful, the orbit $G\cdot x$
is not contained in a prime divisor $D \in {\mathcal D} (X)$.
The intersection $G\cdot x \cap \bigl ( \cup_{{\mathcal D}(X)} D \bigr )$
is a proper Zariski closed subset in $G\cdot x$. By the last assertion
of Lemma~\ref{dimension}, this subset does not contain $G_0^\sigma \cdot x$.
Thus
$G_0^\sigma\cdot x\cap X\setminus \cup_{\mathcal D (X)} D\neq \emptyset$, 
and (ii) follows from ${\rm (i)}$.
\end{proof}

In the remainder,  
$x$ denotes a real point in $X^\circ_G\cap Z$ and $H\subset G$ 
is the stabilizer of $x$. We assume that $\sigma (B) = B$
and $\mu (Z) = Z$. It follows that $\sigma(H)=H$.
Note also that $B\cdot H$ is open in $G$ because the orbit $B\cdot x$ is open
in $X$.

As we recall in Appendix~\ref{localstructure}, 
$T$ acts linearly on $Z$ and the
corresponding characters, say $\gamma_1 ,\ldots,\gamma_r$,
are linearly independent.
These characters are usually called \textsl{spherical roots of $X$}.
Further, we have 
$$
T\cap H=\bigcap_i\,\ker\gamma_i .
$$
Set 
$$
A=T/T\cap H
$$
and let $_2 A \subset A$ be the subgroup of elements of order at most $2$.
Note that any element $t\in T$ can be uniquely written as 
$$
t=t_0t_1, \quad\mbox{ where $t_0 \in T^\sigma_0$ and $\sigma(t_1)=t_1^{-1}$}.
$$
Such a decomposition of $t$ will be referred to 
as the decomposition of $t$ with respect to $\sigma$.

\begin{proposition}\label{auxiliar2}
The $T^{\sigma}_0$-orbits of $\mathbb R Z\cap X^\circ_G$ are in 
one-to-one correspondence with the elements of $_2 A$.
In particular, the number of such orbits does not exceed $2^r$.
\end{proposition}

\begin{proof}
Let $t\in T$ and let $y=t\cdot x$ be a real point.
Then $(\gamma_i\circ\sigma)(t)=\gamma_i(t)$ 
for every spherical root $\gamma_i$ of $X$.
By Lemma~\ref{transformedweight}, it follows that $\gamma_i(t)$ is real-valued.
If $t=t_0t_1$ is
the decomposition of $t$ with respect to $\sigma $,
then we have $\gamma_i(t_1)=\pm 1$. Therefore $t_1 ^2 \in H$.
Assigning to $y\in {\mathbb R} Z \cap X^\circ _G$ the image of $t_1$
in $A$, we get a correctly defined map from the set of $T^\sigma _0$-orbits
on ${\mathbb R} Z\cap X^\circ _G$ to $_2 A$:
$$\alpha :
T^\sigma_0 \setminus ({\mathbb R} Z \cap X^\circ _G) \ \rightarrow \ _2 A .$$
The injectivity of $\alpha $ is obvious. To prove the surjectivity, take
any $t\in T$, such that $t ^2 \in H$.
Then $\gamma _i(t ^2) = 1$, hence $\gamma _i(t) =\pm 1$. It follows
that $t\cdot x$ is a real point. Furthermore,
$\gamma _i(t_0) = 1$ and $\gamma _i(t_1) = \gamma _i (t)$.
Hence $t\cdot H = t_1\cdot H$ and $\alpha (T^\sigma _0\cdot x)
 =t$ mod $T\cap H$. 
\end{proof}

Let $I$ denote a subset of $\{1,\ldots,r\}$ and let $O_I \subset X$
be the corresponding $G$-orbit. Recall that $O_I$ is $\mu$-stable.

\begin{theorem}\label{estimationrealorbits}
\smallbreak\noindent
{\rm (i)}\enspace
Each $G^\sigma_0$-orbit in $\mathbb R O_I$ 
intersects the slice $Z$  in a finite number of $T^\sigma_0$-orbits.
The number of $T^\sigma _0$-orbits in ${\mathbb R}Z \cap O_I$ does not
exceed $2^{r-\left|I\right|}$.
\smallbreak\noindent
{\rm (ii)}\enspace
${\mathbb R} O_I$ contains at most $2^{r-\left|I\right|}$ $G^\sigma$-orbits.
\smallbreak\noindent
{\rm (iii)}\enspace
The total number of $G^\sigma_0$-orbits
in $\mathbb R X$ is smaller than or equal to 
$$
\sum_{k=0}^r 2^k \binom{r}{k}.
$$
\end{theorem}

\begin{proof}
Recall  
that the $G$-orbit closures in $X$ are also strict wonderful $G$-varieties.
Furthermore, 
the rank of the orbit closure ${\rm cl}(O_I)$ equals $r-\left|I\right|$;
see Sect. 3.2 in~\cite{Lu01}.
Thus (i) follows from (ii) of Proposition~\ref{reallocalstructure}, along with the
estimate in Proposition~\ref{auxiliar2}. From (i) we get (ii),
and (iii) is obtained by summing up over all $G$-orbits.
\end{proof}

\begin{example}\smallbreak
\enspace
Let $(\mathbb P^{n})^*$ denote the variety of 
hyperplanes of $\mathbb C^{n+1}$ and 
let $X=\mathbb P^n\times(\mathbb P^{n})^*$ 
be acted on diagonally by $G=PGL_{n+1}(\mathbb C)$.
Suppose $n>1$. Then $X$ is a 
strict wonderful variety of rank $1$. 
The canonical real structure $\mu$ 
is defined by the complex conjugation on each factor of $X$.
Moreover, $G^\sigma_0=G^\sigma=PGL_{n+1}(\mathbb R)$ 
acts on $\mathbb R X$ with two orbits.
\end{example}
\begin{example}
\smallbreak
\enspace
Consider the quadratic form 
$$F(z)=z_1^2+\ldots + z_p^2 -z_{p+1}^2-\ldots-z_{p+q}^2,\ \ 
q\ge p > 0,\ \ p+q >2.$$
The corresponding orthogonal group $G=SO_F$ acts 
on $X= \mathbb P^{m}$ as a subgroup of ${\rm SL}_{m+1}({\mathbb C})$,
where $m = p+q-1$.
Under this action, 
$X$ is a two-orbit $G$-variety. 
The closed $G$-orbit is given by the equation $F=0$.
Again, $X$ is a strict wonderful variety of rank $1$.
Let $\mu:X\rightarrow X$ and 
$\sigma:G\rightarrow G$ be the involutive
mappings defined by complex conjugation.
Then $\mu$ is a $\sigma$-equivariant real structure on $X$.
Note that $\sigma $ defines a split real form of $G$ only for $q=p$
or $q=p+1$.
The real part 
$\mathbb R X$ 
is the real projective space ${\mathbb R \mathbb P}^m $,
on which $G^\sigma_0$ acts with three orbits: $F>0$, $F<0$ and $F=0$. 
\end{example}

\begin{remark}
Starting with a real semisimple symmetric space, 
A. Borel and L. Ji
considered the wonderful completion of the complexified
homogeneous space.  
In this special setting,
the completion is defined over ${\mathbb R}$ in a natural way.  
For the description of real group orbits
on the set of its real points
see ~\cite{BJ}, chapters 5\,-\,7.

\end{remark}

\appendix

\section{Spherical varieties: invariants and local structure}

\subsection{Luna-Vust invariants of spherical homogeneous 
spaces}\label{LV-invariants}

We recollect the definition of the combinatorial 
invariants attached to a given spherical $G$-variety $X$; see ~\cite{LV}.

Let $\mathbb C(X)$ denote the function field of $X$. Then the natural left action of 
$G$ on $X$ yields a $G$-module structure on $\mathbb C(X)$.
\textsl{The weight lattice} $\Lambda^+(X)$ 
is the set of $B$-weights of the $B$-eigenfunctions of $\mathbb C(X)$.
Since $X$ is spherical, 
the $\chi$-weight space of 
$\mathbb C(X)$ is of dimension $1$ for every $\chi\in \Lambda^+(X)$.

Let $\mathcal V(X)$ be the set of $G$-invariant discrete $\mathbb Q$-valued valuations of $\mathbb C(X)$. 
Consider the mapping
$$
\rho: \mathcal V(X)\rightarrow \mathrm{Hom}(\Lambda^+(X),\mathbb Q), \quad v\mapsto (\chi\mapsto v(f_\chi)).
$$
where $f_\chi$ is a $B$-eigenfunction of $\mathbb C(X)$ of weight $\chi$.
The map $\rho$ is injective,
hence one may regard $\mathcal V(X)$ in $\mathrm{Hom}(\Lambda^+(X),\mathbb Q)$.
Further, this cone is convex and simplicial.
The cone $\mathcal V(X)$ is called \textsl{the valuation cone of} $X$; 
see for instance~\cite{Br97}.

Define \textsl{the set of colors} $\mathcal D(X)$ of $X$ 
as the set of $B$-stable, but not $G$-stable prime divisors of $X$. 
This is a finite set equipped with 
two maps, namely, $D\mapsto \rho(v_D)$ and $D\mapsto G_D$ with
$v_D$ (resp. $G_D$) 
being the valuation defined by (resp. the stabilizer in $G$ of) the color $D$.

\textsl{The Luna-Vust invariants of} $X$ 
are given by the triple $\Lambda^+(X)$,  $\mathcal V(X)$, $\mathcal D(X)$.
For two spherical $G$-varieties $X$ and $X'$, 
the equality $\mathcal D(X)=\mathcal D(X')$ means 
that there exists a bijection $\iota:\mathcal D(X)\rightarrow\mathcal D(X')$,
such that $G_D=G_{\iota(D)}$ and $\rho(v_D)=\rho(v_{\iota(D)})$.

\begin{theorem}[\cite{Lo}]
Let $H$ and $H'$ be spherical subgroups of $G$. 
If $H$ and $H'$ have the same Luna-Vust invariants then they are $G$-conjugate.
\end{theorem}

\subsection{Local structure}\label{localstructure}
 
We recall the so-called Local Structure Theorem with special emphasis
on the case of wonderful varieties.

Let $G$ denote a connected reductive algebraic group. 
Fix a Borel subgroup $B$ of $G$ and a maximal torus $T\subset B$ of $G$.

First, consider any normal and irreducible $G$-variety $X$ 
and let $Y$ be a complete $G$-orbit of $X$.
Let $y\in X$ be fixed 
by the Borel subgroup $B^-$ of $G$ opposite to $B$ and containing $T$.
Let $P$ denote the parabolic subgroup of $G$ opposite to the stabilizer 
$G_y$ and containing $T$.
Then $L=P\cap G_y$ is a Levi subgroup of $P$, so that $P = L\cdot P^u$,
where $P^u$ is the unipotent radical of $P$. 
Theorem 1.4 in~\cite{BLV} asserts that 
there exists an affine $L$-variety $Z$,
such that $y\in Z$ and the natural map 
$P^u\times Z\rightarrow P^u \cdot Z$ is an isomorphism.

Suppose now that $X$ is spherical and denote the set of colors of 
$X$ by $\mathcal D(X)$; see Appendix~\ref{LV-invariants}.
Further,  
if $Y$ is the unique closed $G$-orbit in $X$
then $P^u Z$ is the affine set 
$ X\setminus \cup_{\mathcal D(X)} D$ and $Z$ is a spherical $L$-variety. 
In the case of wonderful varieties
Theorem 1.4 
in~\cite{BLV} can be formulated as follows; see ~\cite{Lu96}, Sect. 1.1, 1.2,
and ~\cite{Br97}, Sect.\,2.2\,-\,2.4.

\begin{theorem}
Assume $X$ is a wonderful $G$-variety.
There exists an affine $L$-subvariety $Z$ of $X$ containing $y$ such that
\smallbreak
\noindent{\rm(i)}\enspace
$P^u\times Z\rightarrow X\setminus \cup_{\mathcal D(X)} D: (p,z)\mapsto p.z$ is an isomorphism.
\smallbreak
\noindent{\rm(ii)}\enspace
The derived group of $L$ acts trivially on $Z$ and $Z$ intersects each $G$-orbit of $X$ in one single $T$-orbit.
\smallbreak\noindent
{\rm (iii)} The 
variety $Z$ is the affine space 
of dimension equal to the rank $r$ of $X$ 
Moreover, $Z$ is acted on linearly by 
$T$ and the corresponding $r$ characters of $T$
are linearly independent.
\end{theorem}

\smallskip
Note that (ii) is a consequence
of the configuration of the $G$-orbit closures in a wonderful variety.
To obtain (iii), remark that $Z$ is smooth since 
so is $X$ and thereafter apply (ii) together with Luna Slice Theorem.

Let us now recollect 
how the slice $Z$ is constructed in the case of strict wonderful varieties. 
One may consult ~\cite{BLV} for a general treatment.
Let 
$$
\varphi:X\rightarrow \mathbb P(V)$$
be an embedding of $X$ within the 
projectivization of a finite dimensional $G$-module $V$. 
Thanks to~\cite{P}, we can take $V$ to be simple.
Then $y$ regarded in $\mathbb P(V)$ 
can be written as $[v^-]$ with $v^-$ 
being a $B^-$-eigenvector of $V$ (unique up to a scalar).
Since $L$ is reductive, there exists an $L$-module 
submodule $E\subset V$ 
such that 
$$
V=T_{v^-} G\cdot v^-\oplus E, 
$$ 
where $T_{v^-} G\cdot v^-$ 
stands for the tangent space of the orbit $G\cdot v^-$ at the point $v^-$.
Let $\eta$ be the linear form on $V$ such that 
$\eta(v^-)=1$ and $\eta$ is a $B$-eigenvector.
Let $\mathbb P(\mathbb Cv^-\oplus E)_\eta$ be the open set of $\mathbb P(V)$ on which $\eta$ does not vanish.
Then
$$
Z=\varphi^{-1}\left( \mathbb P(\mathbb Cv^-\oplus E)_\eta\right) .
$$

%-----------------------------------------------------------------------------------------------------------------------------

\end{document}